\documentclass[12pt,psamsfonts, draft]{amsart}
\usepackage{amssymb,amsmath}
\usepackage{amsthm}
\usepackage{mathrsfs}
\usepackage{enumerate, indentfirst}
\usepackage[fleqn,tbtags]{mathtools}
\usepackage{color}
\usepackage{courier}
\usepackage{hyperref}
\usepackage[a4paper,left=2.1cm,right=2.1cm,top=3.2cm,bottom=3.2cm]{geometry}
\usepackage{multicol}
\usepackage{xr}

\usepackage{multicol}
 \newtheorem{theorem}{Theorem}[section]

 \newtheorem{proposition}[theorem]{Proposition}

 \theoremstyle{definition}

 \theoremstyle{remark}
  \numberwithin{equation}{section}

\newcommand{\abs}[1]{\lvert#1\rvert}
\newcommand{\dupN}{\mathbb{N}}

\newcommand{\seq}[1]{(#1_{n})_{n\in\dupN}}

\newcommand{\nen}{n\in\mathbb{N}}

\newcommand{\alg}{\mathscr{A}}

\newcommand{\rpa}{\alg_R^{\dagger}}

\DeclarePairedDelimiterX\sip[2]{(}{)}{#1\,\delimsize\vert\,#2}
\DeclarePairedDelimiterX\siptilde[2]{(}{)_{\!_{\widetilde{A}}}}{#1\,\delimsize\vert\,#2}
\DeclarePairedDelimiterX\sipf[2]{(}{)_{f}}{#1\,\delimsize\vert\,#2}
\DeclarePairedDelimiterX\sipg[2]{(}{)_{g}}{#1\,\delimsize\vert\,#2}
\DeclarePairedDelimiterX\siptw[2]{(}{)_{\tform+\wform}}{#1\,\delimsize\vert\,#2}
\DeclarePairedDelimiterX\set[2]{\{}{\}}{#1\,\delimsize\vert\,#2}
\DeclarePairedDelimiterX\dual[2]{\langle}{\rangle}{#1,#2}

\newcommand{\D}{\mathfrak{X}}

\newcommand{\hil}{\mathscr{H}}

\newcommand{\Ee}{\mathfrak{E}}

\newcommand{\ffi}{\varphi}
\newcommand{\uf}{\mathfrak{u}}

\newcommand{\tf} {\mathfrak{t}}
\newcommand{\wf} {\mathfrak{w}}
\newcommand{\ssf} {\mathfrak{s}}
\newcommand{\Xx}{\mathfrak{X}}

\newcommand{\nullf}{\mathfrak{0}}

\newcommand{\fpx}{\mathcal{F}_+(\Xx)}
\renewcommand*{\phi}{\varphi}

\DeclareMathOperator{\tform}{\mathfrak{t}}
\DeclareMathOperator{\wform}{\mathfrak{w}}

\renewcommand*{\phi}{\varphi}

\newcommand{\Ees}{\mathfrak{E}^{\ast}}

\newcommand{\Kk}{\mathsf{K}}
\newcommand{\Ll}{\mathsf{L}}
\newcommand{\Jj}{\mathsf{J}}

\newcommand{\muw}{\mu_{\wf}}

\newcommand{\muwt}{\mu_{\wf}(\tf)}

\newcommand{\muwn}{\mu_{\wf}^{[n]}}

\newcommand{\muwnpe}{\mu_{\wf}^{[n+1]}}

\begin{document}
\title{Arlinskii's iteration and its applications}

\author[Tam\'as Titkos]{Tam\'as Titkos}

\address{Alfr\'ed R\'enyi Institute of Mathematics\\ Hungarian Academy of Sciences\\ Re\'altanoda u. 13-15.\\
Budapest H-1053\\ Hungary}

\email{titkos.tamas@renyi.mta.hu}
\thanks{The author was supported by the National Research, Development and Innovation Office, NKFIH-104206}


\begin{abstract} 
Several Lebesgue-type decomposition theorems in analysis have a strong relation to the operation called: \emph{parallel sum}. The aim of this paper is to investigate this relation from a new point of view. Namely, using a natural generalization of Arlinskii's approach (which identifies the singular part as a fixed point of a single-variable map) we prove the existence of a Lebesgue-type decomposition for nonnegative sesquilinear forms. As applications, we also show that how this approach can be used to derive analogous results for representable functionals, nonnegative finitely additive measures, and positive definite operator functions. The focus is on the fact that each theorem can be proved with the same completely elementary method.
\end{abstract}

\maketitle
\section{Introduction} The key notion of this paper is the so called \emph{parallel sum}, which was defined first for matrices by Anderson and Duffin \cite{series and parallel addition of matrices} in order to investigate electrical networks. This notion has several generalizations in several aspects, see for example \cite{Bot,Fill-Fishkind,Rickart,Passty}. An exceptionally fruitful is the one given by Ando \cite{Ando-LD}, Pekarev and {\v{S}}mul'jan \cite{PekarevSmuljan}, and many others in the context of bounded positive operators. Ando proved by means of parallel addition that a Lebesgue-type decomposition of bounded positive operators always exists.  Besides of its own interest, this type of decomposition plays a crucial role by investigating some extremal problems regarding the order structure of the positive cone (see for example \cite{EL,ggj}). We have to mention also Simon's fundamental paper \cite{Simon}, in which he suggested that there may be a noteworthy connection between the classical Lebesgue decomposition of measures, the canonical decomposition of densely defined quadratic forms, and the decomposition of states of $C^*$-algebras.

The aim of this paper, on the one hand, is to present an elementary method to obtain the Lebesgue-type decomposition in the context of nonnegative sesquilinear forms \cite{lebdec}. By mimicking the method of Arlinskii, we identify the singular part as a fixed point of a single-variable map. On the other hand, our second aim is to show that how this elementary method can be used to derive the corresponding decompositions of representable functionals (cf. \cite{G,kosaki,tarcsay-funkleb}), nonnegative finitely additive measures (cf. \cite{Darst,Rao}), and positive definite operator functions (cf. \cite{as,s1}), demonstrating that these results have a common root. The focus is on the fact that despite of the structural differences, all of the decomposition theorems mentioned above can be proved simultaneously with the same completely elementary method.

\section{Preliminaries}
To begin with, let us recall some basic notions. A nonnegative sesquilinear form on the complex linear space $\D$ is a mapping $\tf:\D\times\D\to\mathbb{C}$, which is linear in its first, anti-linear in its second argument, and the corresponding quadratic form is nonnegative, i.e. $\tf[x]:=\tf(x,x)\geq0$ for all $x\in\D$. Recall that such a quadratic form satisfies the parallelogram law and the polarization identity. Since every sesquilinear form in this paper is assumed to be nonnegative, we write shortly: \emph{form}. We write $\tf\leq\wf$, if $\tf[x]\leq\wf[x]$ for all $x\in\D$. The partially ordered set of forms will be denoted by $\fpx$. A sequence $(\tf_n)_{\nen}$ of forms is monotone decreasing if $m\leq n$ implies that $\tf_n\leq\tf_m$. If $(\tf_n)_{n\in\mathbb{N}}$ is a monotone decreasing sequence, then the pointwise limit ($\tf[x]:=\lim\limits_{n\in\mathbb{N}}\tf_n[x]$ for all $x\in\Xx$) exists and defines a form. We will denote this shortly by $\tf_n\downarrow\tf$. We remark that this limit is equal to the order infimum of the set $\{\tf_n\,|\,n\in\mathbb{N}\}$. Similarly, the pointwise limit of an upper bounded monotone increasing sequence ($m\leq n$ implies $\tf_m\leq\tf_n\leq\ssf$ for some form $\ssf$) is a form, and $\tf=\lim\limits_{n\in\mathbb{N}}\tf_n=\sup\limits_{n\in\mathbb{N}}\tf_n\leq\ssf$. We will use the notation $\tf_n\uparrow\tf$ in this case.\\
The form $\tf$ is \emph{$\wf$-dominated} if there exists an $\alpha\geq0$ such that $\tf\leq \alpha\wf$. The form $\tf$ is \emph{$\wf$-almost dominated}, if $\tf_n\uparrow\tf$ holds for some sequence $\seq{\tf}$ of $\wf$-dominated forms. We say that $\tf$ is \emph{$\wf$-closable} if for any sequence $(x_n)_{\nen}$ in $\Xx$ the following implication holds
\begin{align}\label{t is w-closable}
\big(\tf[x_n-x_m]\rightarrow 0\quad\mbox{and}\quad\wf[x_n]\rightarrow 0\big)\qquad\Longrightarrow\qquad\tf[x_n]\rightarrow 0.
\end{align}
We will use frequently the following important fact (for the proof see \cite[Theorem 3.8]{lebdec})
\begin{align}\label{clad}
\tf~\mbox{is}~\wf\mbox{-closable}\qquad\Longleftrightarrow\qquad\tf~\mbox{is}~\wf\mbox{-almost dominated}.
\end{align}
The form $\tf$ is $\wf$-\emph{singular} if for each form $\ssf$ on $\D$ the inequalities $\ssf\leq\tf$ and $\ssf\leq\wf$ imply that $\ssf=\mathfrak{0}$ (where $\nullf$ denotes the zero form). A decomposition of $\tf$ into $\wf$-almost dominated and $\wf$-singular parts is called $\wf$-Lebesgue decomposition. To see that this type of decomposition is not unique in general we refer the reader to \cite[Theorem 4.4]{lebdec}. An analogous decomposition of not necessarily nonnegative sesquilinear forms can be found in the recent paper of Di Bella and Trapani \cite[Section 4]{Trapani}.\\

We remark that the idea of decomposing (densely defined) sesquilinear forms into regular and singular parts goes back to Simon \cite{Simon}, and there are recent results in the same spirit in the theory of partial differential equations (see for example \cite{ter Elst 1,ter Elst 2,Vogt}).\\

As was mentioned above, the central notion of our approach is the parallel sum of forms which was introduced by Hassi, Sebesty\'en, and de Snoo in their fundamental paper \cite{lebdec}. The definition and the properties of parallel addition are given in the following proposition  (for the details see \cite[Proposition 2.2 and Lemma 2.3]{lebdec}).

\begin{proposition}\label{parsumproperty}
For $\tf,\wf\in\fpx$ the quadratic form of the parallel sum $\tf:\wf$ is
    \begin{align*}
        (\tf:\wf)[x]:=\inf\limits_{y\in\D}\big\{\wf[y+x]+\tf[y]\big\},\qquad(x\in\D).
    \end{align*}
Furthermore, parallel addition satisfies the following properties
\begin{itemize}
\begin{multicols}{2}
\item[(a)] $\tf:\wf=\wf:\tf\leq\tf$,
\item[(b)] $(\tf:\wf):\ssf=\tf:(\wf:\ssf)$,
\item[(c)] $\tf\leq\ssf~\Rightarrow~\tf:\wf\leq\ssf:\wf$,
\item[(d)] $\lambda\tf:\mu\tf=\textstyle{\frac{\mu\lambda}{\mu+\lambda}}\tf$\quad($\lambda,\mu>0$),
\item[(e)] $(\tf:\wf)+(\ssf:\uf)\leq(\tf+\ssf):(\wf+\uf)$,
\item[(f)] $\tf_n\downarrow\tf,~\wf_n\downarrow \wf~\Rightarrow~\tf_n:\wf_n\downarrow\tf:\wf$.
\end{multicols}
\end{itemize}
\end{proposition} 
We close this section with the following simple characterization of $\wf$-singularity
\begin{align}\label{t:w=0}
\tf~\mbox{is}~\wf\mbox{-singular}\qquad\Longleftrightarrow\qquad\tf:\wf~\mbox{is the zero form}.
\end{align}
Indeed, for any $\uf\in\fpx$ for which $\uf\leq\tf$ and $\uf\leq\wf$ hold, we have
$\nullf=\tf:\wf\geq\uf:\uf=\frac{1}{2}\uf\geq\nullf$
according to Proposition \ref{parsumproperty} (a) and (d). The converse implication is obvious.

\section{
Arlinskii's iteration}
In order to investigate von Neumann's well known result \cite[Satz 18.]{Neumann}, Arlinskii presented an iteration scheme for bounded positive operators in \cite[Theorem 5.4]{Arlinskii}. In this section we will apply this iteration in the context of forms. 
Let $\wf$ be a form on $\Xx$ and define the map 
\begin{equation}\label{singfix}
\muw:\fpx\to\fpx;\qquad\qquad\muwt:=\tf-\tf:\wf.
\end{equation}
Observe immediately that $\tf$ is a fixed point of $\muw$ if and only if $\tf$ is $\wf$-singular. Indeed, 
\begin{equation*}
\muwt=\tf\quad\Longleftrightarrow\quad\tf:\wf=0\quad\Longleftrightarrow\quad\tf\perp\wf.
\end{equation*} To find a fixed point of a map, it is convenient to use the following scheme:
let $\mu_{\wf}^{[0]}(\tf)=\tf$ and for $n\geq1$ denote by $\muwn$ the $n$th iteration of $\muw$, that is,  \begin{equation}\label{sorozat}
\muwn(\tf)=\muw\big(\mu_{\wf}^{[n-1]}(\tf)\big)=\mu_{\wf}^{[n-1]}(\tf)-\mu_{\wf}^{[n-1]}(\tf):\wf.
\end{equation}
Now we offer a completely elementary proof of the existence of a Lebesgue-type decomposition.
\begin{theorem}\label{THM}
Let $\tf$ and $\wf$ be forms on $\Xx$. Then a $\wf$-Lebesgue decomposition of $\tf$ exists, i.e. $\tf$ splits into $\wf$-almost dominated and $\wf$-singular parts.
\end{theorem}
\begin{proof}
Since $\mu_{\wf}(\ssf)\leq\ssf$ for all $\ssf\in\fpx$, the sequence defined in \eqref{sorozat} is monotone decreasing. Consequently, there exists a form (denoted by $\tau_{\wf}(\tf)$) such that $\muwn(\tf)\downarrow\tau_{\wf}(\tf)$. According to Proposition \ref{parsumproperty} (f), we have on the one hand that $\muwn(\tf):\wf\downarrow\tau_{\wf}(\tf):\wf$. On the other hand, 
\begin{align*}
\lim\limits_{n\to\infty}\left(\muwn(\tf):\wf\right)=\lim\limits_{n\to\infty}\left(\muwn(\tf)-\muwnpe(\tf)\right)=\nullf.
\end{align*}
Putting these two facts together, we get that $\tau_{\wf}(\tf):\wf=\nullf$, or equivalently, $\tau_{\wf}(\tf)$ is $\wf$-singular.

What is left is to show that $\tf-\tau_{\wf}(\tf)$ is $\wf$-almost dominated.  To see this, first observe that $\tf-\muwn(\tf)\uparrow\tf-\tau_{\wf}(\tf)$. Furthermore, $\tf-\muwn(\tf)$ can be written as a telescopic sum which satisfies
\begin{equation}\label{teleszkop}
\tf-\muwn(\tf)=\sum\limits_{k=0}^{n-1}\mu_{\wf}^{[k]}(\tf):\wf\leq n\wf
\end{equation}
according to Proposition \ref{parsumproperty} (a), and hence the proof is complete.
\end{proof}

%

\section{Applications}
In this section we are going to show some applications of Theorem \ref{THM}. Of course, if $\mathscr{H}$ is a Hilbert space with inner product $(\cdot\,|\,\cdot)$, and we consider the nonnegative sesquilinear form $$\tf_A(x,y):=(Ax\,|\,y)\qquad(x,y\in\mathscr{H})$$
corresponding to a bounded positive operator $A$, then we recover the original setting of Arlinskii. Since parallel addition can be defined in the same spirit for representable functionals, for nonnegative finitely additive measures, and for positive definite operator functions, we can obtain analogous decompositions in each case. We will also make some remarks on the connection between these decompositions and the corresponding classical results.

\subsection{Representable functionals}
Our constant reference in this subsection is the fundamental book of Palmer \cite[Chapter 9]{palmer}. For the sake of completeness, we present here the basic notions and notations.
Let $\alg$ be a unital $^*$-algebra. A linear functional $w:\alg\to\mathbb{C}$ is \emph{positive} if $w(a^{*}a)\geq0$ for all $a\in\alg$. For the positive functionals $w$ and $v$ we write $w\leq v$ if $v-w$ is positive. A linear functional $w$ on $\alg$ is called \emph{representable} if there is a $^*$-representation $\pi_w$ of $\alg$ in a Hilbert space $\hil_w$ and a vector $\zeta_w\in\hil_w$ such that
\begin{align}\label{E: representing}
w(a)=\sip{\pi_w(a)\zeta_w}{\zeta_w},\qquad (a\in\alg).
\end{align}
A positive functional is \emph{Hilbert bounded} if there exists a constant $C\geq0$ such that \begin{align}\label{hb}
\abs{w(a)}^2\leq Cw(a^*a),\qquad(a\in\alg).
\end{align}
The least possible $C$ satisfying \eqref{hb} is called the \emph{Hilbert bound} of $w$, and is denoted by $\|w\|_H$. We remark that each representable functional is Hilbert bounded and positive. The partially ordered set of representable functionals will be denoted by $\rpa$. If $u$ is a Hilbert bounded positive functional which satisfies $u\leq v$ for some $v\in\rpa$, then $u$ and $v-u$ are both representable (with $\|u\|_H\leq\|v\|_H$ and $\|v-u\|_H\leq\|v\|_H$). If $\seq{w}$ is a monotone decreasing sequence in $\rpa$ (that is, if $w_j\geq w_k\geq 0$ if $j\leq k$) then the pointwise limit $w$ belongs to $\rpa$ and is equal to the greatest lower bound of $\seq{w}$ in $\rpa$.

Let us define the notions of strong absolute continuity and singularity of representable functionals. Using Gudder's terminology, we say that $w$ is \emph{strongly $v$-absolute continuous} if \begin{equation}\label{strong ac}
\Big(v(a_n^*a_n)\to0\quad\textrm{and}\quad w((a_n-a_m)^*(a_n-a_m))\to0\Big)\quad\Longrightarrow \qquad w(a_n^*a_n)\to0.
\end{equation}
We say that $w$ is \emph{$v$-singular} if $u\leq w$ and $u\leq v$ imply $u=0$ for any $u\in\rpa$. The aim of this subsection is to prove that every $w\in\rpa$ can be decomposed into strongly $v$-absolute continuous and $v$-singular parts.
Since every representable functional defines a form on $\alg$, namely
$$w\mapsto\tf_w;\qquad \tf_w(a,b):=w(b^*a),\quad(a,b\in\alg)$$
it seems to be easy to apply Theorem \ref{THM}. But we have to keep in mind that not every form on $\alg$ is induced by a representable functional. If we are able to guarantee for any pair $v,w\in\rpa$ that there exists a $u\in\rpa$ such that $\tf_u=\tf_w-\tf_w:\tf_v=\mu_{\tf_w}(\tf_v)$, then we can mimicking Arlinskii's approach. Using Tarcsay's recent result we can show that this is indeed the case. 

Consider the GNS triplets $(\hil_w,\pi_w,\zeta_w)$ and $(\hil_v,\pi_v,\zeta_v)$ associated with $w$ and $v$, respectively. Let $\pi$ stand for the direct sum of $\pi_w$ and $\pi_v$, and let $P$ be the orthogonal projection onto the ortho-complement of the following $\pi$-invariant subspace
\begin{equation*}
\set{\pi_w(a)\zeta_w\oplus\pi_v(a)\zeta_v}{a\in\alg}\subseteq\hil_w\oplus\hil_v.
\end{equation*}
Tarcsay proved in \cite[Theorem 5.1]{Tarcsay_parallel} that if we define the parallel sum $w:v$ by
\begin{equation*}
        (w:v)(a):=\sip{\pi(a)P(\zeta_w\oplus0)}{P(\zeta_w\oplus0)},\qquad (a\in\alg),
\end{equation*}
then $w:v$ is a representable functional which satisfies
\begin{equation}\label{E:f:g_q}
        (w:v)(a^*a)=\inf\set{w((a-b)^*(a-b))+v(b^*b)}{b\in\alg},\qquad (a\in\alg).
\end{equation}
We emphasize here that $\tf_{w:v}=\tf_w:\tf_v$ holds according to \eqref{E:f:g_q}. Summarizing these observations, we conclude that the function $\widetilde{\mu_v}$ defined by $\widetilde{\mu_v}(w):=w-w:v$ maps $\rpa$ into $\rpa$. Now we establish a Lebesgue-type decomposition in the $^*$-algebra context. The same result has been obtained in \cite[Theorem 3.3]{tarcsay-funkleb} with an essentially different proof.

\begin{theorem}\label{LD rep}
Let $w$ and $v$ be representable functionals on $\alg$. Then $w$ splits into strongly $v$-absolute continuous and $v$-singular parts.
\end{theorem}
\begin{proof}
Consider the forms $\tf_w$ and $\tf_v$.  Theorem \ref{THM} says that $\tf_w$ can be written as 
$$\tf_w=\big(\tf_w-\tau_{\tf_v}(\tf_w)\big)+\tau_{\tf_v}(\tf_w),$$
where $\tf_w-\tau_{\tf_v}(\tf_w)$ is $\tf_v$-almost dominated and $\tau_{\tf_v}(\tf_w)$ is $\tf_v$ singular. According to the previous observations, both forms are induced by representable functionals, say $w_{r}$ and $w_{s}$, respectively. Comparing \eqref{t is w-closable} and \eqref{strong ac} we obtain that $w_{r}$ is strongly $v$-absolute continuous if and only if $\tf_{w_{r}}$ is $\tf_{v}$-closable, or equivalently $\tf_{w_{r}}=\big(\tf_w-\tau_{\tf_v}(\tf_w)\big)$ is $v$-almost dominated according to \eqref{clad}. Singularity of $w_{s}$ and $v$ follows from the $\tf_v$-singularity of $\tf_{w_{s}}=\tau_{\tf_v}(\tf_w)$. Indeed, assume that there exists a nonzero $u\in\rpa$ such that $u\leq w_{s}$ and $u\leq v$. In this case, $\tau_{\tf_v}(\tf_w):\tf_v\geq\tf_{u}:\tf_{u}=\frac{1}{2}\tf_u$, which contradicts \eqref{t:w=0}.
\end{proof}
\noindent\textbf{Remark.} This type of decomposition has been investigated by Kosaki \cite{kosaki} in the following important special case: $\alg$ is a von Neumann algebra, $w$ is a normal positive functional, and $v$ is a faithful normal state. (Recall that positive functionals on $C^*$-algebras are representable according to \cite[Theorem 11.3.3]{palmer}.) Kosaki also presented an example in \cite[Section 10]{kosaki} which demonstrates that the Lebesgue-type decomposition in Theorem \ref{LD rep} is not unique in general.
\subsection{Finitely additive measures}

In this section we show that how the Lebesgue-Darst decomposition of nonnegative finitely additive measures can be obtained by our approach.

We will follow the terminology of \cite{Rao}. Let $\mathcal{F}$ be a field of subsets of a set $\Omega$. A set function $\mu:\mathcal{F}\to\mathbb{R}$ is called a \emph{charge}  if $0\leq\mu(A)\leq\mu(\Omega)<+\infty$ for all $A\in\mathcal{F}$, and is additive, i.e. $\mu(A\cup B)=\mu(A)+\mu(B)$ whenever $A$ and $B$ are disjoint elements of $\mathcal{F}$. We say that $\mu$ is $\nu$-absolute continuous, if for every
$\varepsilon>0$ there exists $\delta
>0$ such that $\mu(A)<\varepsilon$, whenever $A\in\mathcal{F}$ and $\nu(A)<\delta$. The charge $\mu$ is called $\nu$-singular if for every charge $\eta$ the inequalities $\eta\leq\mu$ and $\eta\leq\nu$ imply $\eta=\theta$, where $\theta$ is the zero charge, and the symbol $\leq$ refers to the partial order
\begin{align*}
\mu\leq\nu\qquad\Longleftrightarrow\qquad\forall A\in\mathcal{F}:~~\mu(A)\leq\nu(A).
\end{align*}
A decomposition of $\mu$ into $\nu$-absolute continuous and $\nu$-singular parts is called a Lebesgue decomposition of $\mu$ with respect to $\nu$. We are going to show that such a decomposition always exists. Let us denote the linear space of complex $\mathcal{F}$-step functions with $\mathcal{S}$, and consider the form $\tf_{\mu}$ induced by the charge $\mu$
\begin{align}\label{induced form}
\tf_{\mu}(\varphi,\psi):=\int\limits_{\Omega}\varphi\cdot\overline{\psi}~\mathrm{d}\mu,\qquad(\ffi,\psi\in\mathcal{S}).
\end{align}
Here of course, the integral is the elementary integral, i.e. it is just a finite sum. Similarly as for representable functionals, not every form on $\mathcal{S}$ is induced by a charge, so we have to explain that why Arlinskii's iteration is applicable. 
If a form $\tf$ on $\mathcal{S}$ is given, the natural way to define a set function corresponding to $\tf$ is the following: 
\begin{align*}\vartheta_{\tf}(A):=\tf[\chi_A],\qquad(A\in\mathcal{F})
\end{align*}
where $\chi_A$ denotes the characteristic function of $A$. It is easy to see that if $\tf[\varphi]=\tf[|\varphi|]$ holds for all $\varphi\in\mathcal{S}$, then $\vartheta_{\tf}$ is additive. Indeed, for any disjoint sets $A,B\in\mathcal{F}$ we have
\begin{align*}
\tf[\chi_A]+\tf[\chi_B]=\frac{1}{2}(\tf[\chi_A+\chi_B]+\tf[\chi_A-\chi_B])=\frac{1}{2}(\tf[|\chi_A+\chi_B|]+\tf[|\chi_A-\chi_B|])=\tf[\chi_A+\chi_B]
\end{align*}
according to the parallelogram law. Furthermore, we see from \eqref{induced form} that the converse implication is also true, i.e. $\tf_{\mu}[\ffi]=\tf_{\mu}[|\ffi|]$ for all $\ffi\in\mathcal{S}$. Let us define the parallel sum of two charges by the identity
\begin{align}\label{ps content}
(\mu:\nu)(A):=(\tf_{\mu}:\tf_{\nu})[\chi_A],\qquad(A\in\mathcal{F}).
\end{align}
Let $\ffi=\sum\limits_{j=1}^n\lambda_k\chi_{{}_{A_j}}\in\mathcal{S}$ be fixed, where $A_i\cap A_j=\emptyset$ whenever $i\neq j$. Now define the function $m_{\ffi}:\Omega\to\mathbb{C}$ corresponding to $\ffi$ by
\begin{align}
m_{\ffi}(x):=\sum\limits_{j=1}^n\frac{|\lambda_j|}{\lambda_j}\chi_{{}_{A_j}}(x)+\chi_{{}_{X\setminus \bigcup_{j=1}^nA_j}}(x),\quad (x\in\Omega).
\end{align}
Since $|m_{\ffi}(x)|=1$ for all $x\in\Omega$, the multiplication by $m_{\ffi}$ is a bijection on $\mathcal{S}$. Furthermore, since both $\tf_{\mu}$ and $\tf_{\nu}$ are derived from charges, we obtain by elementary calculation that
\begin{align*}
\tf_{\nu}[\psi]=\tf_{\nu}[m_{\ffi}\psi]\quad\mbox{and}\quad\tf_{\mu}[\ffi+\psi]=\tf_{\alpha}[|\ffi|+m_{\ffi}\psi]
\end{align*}
hold for all $\psi\in\mathcal{S}$, and hence $(\tf_{\mu}:\tf_{\nu})[\ffi]=(\tf_{\mu}:\tf_{\nu})[|\ffi|].$ Since $\tf_{\mu}:\tf_{\nu}\leq\tf_{\nu}$, the set function $\mu:\nu$ defined in \eqref{ps content} satisfies $\mu:\nu\leq\nu$, and is a charge, indeed. This guarantees that there exists a charge $\vartheta$ such that $\tf_{\vartheta}=\tf_{\mu}-\tf_{\mu}:\tf_{\nu}$, namely $\vartheta=\mu-\mu:\nu$, and hence we can apply Arlinskii's iteration.
\begin{theorem}\label{LebesgueDarst}
Let $\mu$ and $\nu$ be charges on $\mathcal{F}$. Then $\mu$ splits into strongly $\nu$-absolute continuous and $\nu$-singular parts.
\end{theorem}
\begin{proof}
Take the forms $\tf_{\mu}$ and $\tf_{\nu}$, and apply Theorem \ref{THM}. Since the setwise limit of additive set functions is additive, we obtain a decomposition $\mu=\mu_{reg}+\mu_{sing}$, where $\tf_{\mu_{reg}}=\tau_{\tf_{\nu}}(\tf_{\mu})$. The $\nu$-singularity of $\mu_{sing}$ follows from $\tf_{\nu}$-singularity of $\tf-\tau_{\tf_{\nu}}(\tf_{\mu})$. To prove the $\nu$-absolute continuity of $\mu_{reg}$ observe from \eqref{teleszkop} 
that $\mu_{reg}$ is a setwise limit of a monotone increasing sequence $\seq{\mu}$ with $\mu_n\leq n\nu$. Assume indirectly that $\mu_{reg}$ is not $\nu$-absolute continuous. Then there exists an
$\varepsilon>0$ and a sequence $(A_n)_{n\in\mathbb{N}}$ of measurable sets such that $\nu(A_n)\rightarrow0$ as $n\to\infty$ and $\mu_{reg}(A_n)>\varepsilon$ for every
$\nen$. We can fix a $j\in\mathbb{N}$ which satisfies the first inequality in
\begin{align*}
\mu_{reg}(X)-\frac{\varepsilon}{2}<\mu_j(X)=\mu_j(X\setminus A_k)+\mu_j(A_k)\leq\mu_{reg}(X\setminus A_k)+j\nu(A_k),\qquad(k\in\mathbb{N}).
\end{align*}
But this implies $\mu(A_k)\leq j\nu(A_k)+\frac{\varepsilon}{2}<\varepsilon$ if $k$ is big enough, which is a contradiction.
\end{proof}
\noindent\textbf{Remark.} This type of decomposition is unique in contrast with the one presented in Theorem \ref{THM}. Since the absolute continuity and singularity concepts used here are equivalent to those used in \cite{Darst}, the decomposition in Theorem \ref{LebesgueDarst} is just the Lebesgue-Darst decomposition in the nonnegative case. For a completely different Hilbert space theoretic approach (which uses the same charge-form correspondence), we refer the reader to \cite[Section 3.]{stt1}.

\subsection{Positive definite operator functions} Szyma\'nski in \cite{s1} developed a general dilatation theory by means of nonnegative sesquilinear forms. He pointed out that closability in the operator function set-up has a spacial meaning. In order to formulate his result we need first some technical details. For the sake of simplicity, we choose the setting of \cite[Section 8]{lebdec} instead of \cite[Section 2]{s1}.
Let $\Ee$ be a complex Banach space, and denote by $\mathbf{B}(\Ee,\Ees)$ the set of bounded linear operators from $\Ee$ to $\Ees$. A duality between $\Ee$ and $\Ee^{\ast}$ is a mapping $\langle\cdot,\cdot\rangle:\Ee\times\Ees\to\mathbb{C}$ which is linear in its first, conjugate linear in its second variable.
Let $S$ be a non-empty set, and let $\Xx$ be the complex linear space of all functions on $S$ with values in $\Ee$ with finite support. We say that the function $\Kk:S\times S\to \mathbf{B}(\Ee,\Ees)$ is a \emph{positive definite operator function}, or shortly a \emph{kernel} on $S$ if the following equality defines a form on $\Xx$
\begin{align}\label{kernelform}
\wf_{\Kk}(f,g):=\sum\limits_{s,t\in S}\big\langle f(t),\Kk(s,t)g(s)\big\rangle,\qquad(f,g\in\Xx).
\end{align}
For the kernels $\Kk$ and $\Ll$ we write $\Kk\prec\Ll$ if $\wf_\Kk\leq\wf_\Ll$. We say that $\Kk$ is $\Ll$-closable if $\wf_{\Kk}$ is $\wf_{\Ll}$-closable. Similarly, $\Kk$ is $\Ll$-singular if $\wf_{\Kk}$ is $\wf_{\Ll}$-singular. A decomposition into $\Ll$-closable and $\Ll$-singular parts is called $\Ll$-Lebesgue decomposition.
As in the previous cases, we want to apply Arlinskii's iteraton. To do so, we have to know that there is a kernel $\Jj$ such that $\wf_{\Jj}=\wf_{\Kk}:\wf_{\Ll}$. This is indeed the case, because Hassi, Sebesty\'en, and de Snoo proved in \cite[Lemma 7.1]{lebdec} that if $\Kk$ is a kernel and $\wf$ is a form satisfying $\wf\leq\wf_{\Kk}$, then there exists a unique kernel $\Ll$ such that $\Ll\prec\Kk$ and $\wf=\wf_{\Ll}$ in the sense of \eqref{kernelform}. The kernel $\Jj$ is called the parallel sum of $\Kk$ and $\Ll$, and is denoted by $\Kk:\Ll$.

\begin{theorem}\label{Kernel Lebesgue}
Let $\Kk$ and $\Ll$ be kernels. Then $\Kk$ splits into $\wf$-closable and singular parts.\end{theorem}
\begin{proof}
Take the forms $\wf_{\Kk}$ and $\wf_{\Ll}$ and apply Theorem \ref{THM}. According to the previous observations, we have a kernel $\Kk_{r}$ such that $\wf_{\Kk_{r}}=\wf_{\Kk}-\tau_{\wf_{\Ll}}(\wf_{\Kk})$. Since $\wf_{\Ll}-\tau_{\wf_{\Ll}}(\wf_{\Kk})$ is $\wf_{\Ll}$-almost dominated, or equivalently, $\wf_{\Ll}$-closable, we obtain that $\Kk_{r}$ is $\Ll$-closable. The $\Ll$-singularity of $\Kk_{s}:=\Kk-\Kk_{r}$ follows form the singularity of their forms.
\end{proof}
\noindent\textbf{Remark.} As was mentioned at the beginning of this subsection, closability has a special meaning in this set-up. Namely, Szyma\'nski proved in \cite[Theorem (3.5)]{s1} that $\Kk$ is $\Ll$-closable if and only if $\Kk$ has a closed dilation (i.e. a closed operator) acting on an auxiliary Hilbert space associated to $\Ll$. Consequently, Theorem \ref{Kernel Lebesgue} can be viewed as a decomposition into dilatable and singular parts.

\end{document}